\documentclass[11pt]{amsart}

\usepackage{amsmath,amssymb,amsthm,mathtools}
\usepackage{geometry}
\usepackage{microtype}
\usepackage[hidelinks]{hyperref}
\hypersetup{
  pdftitle={Ergodic treeable equivalence relations are richly ergodic},
  pdfauthor={Wei Dai}
}

\theoremstyle{plain}
\newtheorem{theorem}{Theorem}[section]
\newtheorem{corollary}[theorem]{Corollary}
\newtheorem{lemma}[theorem]{Lemma}
\newtheorem{claim}[theorem]{Claim}
\newtheorem{proposition}[theorem]{Proposition}

\theoremstyle{definition}
\newtheorem{definition}[theorem]{Definition}

\newtheorem{question}[theorem]{Question}

\newcommand{\N}{\mathbb{N}}
\newcommand{\Z}{\mathbb{Z}}
\newcommand{\F}{\mathbb{F}}
\newcommand{\cS}{\mathcal{S}}

\title{Ergodic treeable equivalence relations are richly ergodic}
\author{Wei Dai}
\address{School of Mathematical Sciences and LPMC, Nankai University,
Tianjin 300071, P.R. China}
\email{weidai@nankai.edu.cn}
\date{}

\begin{document}

\begin{abstract}
We study the density of ergodic subrelations in the canonical Polish space
of subequivalence relations of a probability-measure-preserving countable
Borel equivalence relation. Our main result establishes rich ergodicity
for ergodic treeable equivalence relations. The proof combines a strongly increasing
exhaustion with relative ergodic hyperfinite decomposition. As an
application, for every $2\leq r<\infty$, we show that the product of the
dyadic odometer action of $\Z$ with the Bernoulli action of $\F_r$ has a
richly ergodic orbit relation, whereas the Bernoulli action of
$\Z\times\F_r$ does not. Thus rich ergodicity of a free ergodic action is
not determined by the acting group.
\end{abstract}

\maketitle

\section{Introduction and main results}

Throughout, unless otherwise stated, ambient probability spaces are
standard and nonatomic, and all countable Borel equivalence relations
are probability-measure-preserving (p.m.p.). All relations,
transformations, inclusions, and equalities are understood modulo null
sets unless a pointwise assertion is explicitly made. We write
$\N=\{0,1,2,\ldots\}$.

For a p.m.p. countable Borel equivalence relation $E$ on $(X,\mu)$,
Kechris equips the space $\cS(E)$ of its p.m.p. subequivalence relations
with a canonical Polish topology \cite{Kechris}. An ergodic relation $E$
is \emph{richly ergodic} if its ergodic subrelations are dense in
$\cS(E)$. Since the ergodic subrelations form a $G_\delta$ set
\cite[Theorem~10.1]{Kechris}, this is equivalent to saying that the
generic subrelation of $E$ is ergodic. Kechris asks which ergodic
relations have this property \cite[Problem~11.1]{Kechris}; see also
\cite[Section~8]{KechrisSurvey} for an account of the initial results.

Several positive results suggest that treeability should be sufficient
for rich ergodicity. Every ergodic hyperfinite relation is richly
ergodic \cite[Proposition~11.7]{Kechris}. Kechris also proves rich
ergodicity for free mixing actions of a countably infinite free product
$\Gamma_1*\Gamma_2*\cdots$ of nontrivial countable groups
\cite[Proposition~8.6]{KechrisSurvey}. Taking every factor to be $\Z$
gives, in particular, the orbit relations of free mixing actions of
$\F_\infty$. Thus rich ergodicity is known both for hyperfinite relations
and for this class of non-hyperfinite treeable relations. Our first
theorem shows that treeability alone suffices.

\begin{theorem}\label{thm:treeable-rich}
Every ergodic treeable p.m.p. countable Borel equivalence relation is
richly ergodic.
\end{theorem}

In particular, every free ergodic p.m.p. action of $\F_r$, for
$2\leq r<\infty$, has a richly ergodic orbit relation, regardless of
whether the action is mixing. More generally, the conclusion requires
only treeability of the equivalence relation, not a prescribed presentation by a
group action.

After treating the hyperfinite case separately, the proof of
Theorem~\ref{thm:treeable-rich} constructs an exhaustion of a treeing while reserving a positive-measure collection of unused edges
at each stage. Compression and conjugation allow these edges to supply
an ergodic automorphism perpendicular to that stage. The relative
Miller--Tserunyan decomposition theorem \cite[Theorem~6.13]{MT} keeps a
prescribed ergodic base fixed. This produces the strongly increasing
exhaustion required by Kechris's criterion, including when the ambient
relation has finite cost.

There is also a source of obstructions to rich ergodicity. Following
Gaboriau--Tucker-Drob \cite{GTD}, a relation $E$ is
\emph{non-approximable} if every increasing exhaustion
$E=\bigcup_n E_n$ has a stage satisfying
$E_n\upharpoonright B=E\upharpoonright B$ on some positive-measure
Borel set $B$. Kechris proves that an ergodic non-approximable relation
cannot be richly ergodic \cite[Proposition~11.3]{Kechris}. Thus the
non-approximability results give a systematic way to obtain negative
examples.

For direct products, Gaboriau--Tucker-Drob show that a free p.m.p.
action of $H\times K$, with $H$ and $K$ infinite and finitely generated,
is non-approximable when the $H$-restriction is strongly ergodic and the
$K$-restriction is ergodic \cite[Theorem~1.1]{GTD}. In particular, the
Bernoulli action of $\Z\times\F_r$, $2\leq r<\infty$, is not richly
ergodic: its $\F_r$-restriction is strongly ergodic and its
$\Z$-restriction is ergodic. This raises the question whether the same
group can also admit a free ergodic action with a richly ergodic orbit
relation. Our second theorem gives an affirmative answer.

\begin{theorem}\label{thm:rich_and_non_rich}
Let $2\leq r<\infty$. The group $\Z\times\F_r$ admits
\begin{enumerate}
\renewcommand{\labelenumi}{(\roman{enumi})}
\item a free ergodic p.m.p. action whose orbit equivalence relation is
richly ergodic; and
\item a free ergodic p.m.p. action whose orbit equivalence relation is
not richly ergodic.
\end{enumerate}
\end{theorem}

The action in part~(i) is the product of the dyadic odometer action of
$\Z$ with the Bernoulli action of $\F_r$. We express its orbit relation
as an increasing union of ergodic treeable subrelations, using finite
coordinate flips to twist one free generator. Theorem~\ref{thm:treeable-rich}
and the preservation of rich ergodicity under increasing unions
\cite[Proposition~11.6]{Kechris} then give the conclusion. Part~(ii) is
the Bernoulli action discussed above. Hence rich ergodicity of a free
ergodic action is not determined by the acting group. The positive
example is consistent with the non-approximability criterion: in this
separate-coordinate product action, neither factor acts ergodically on
the whole product space.

The paper is organized as follows. Section~2 recalls the canonical Polish
topology on $\cS(E)$ and the basic results on rich ergodicity,
non-approximability, and strongly increasing exhaustions. Section~3
collects the necessary facts about cost and treeability, including the
relative Miller--Tserunyan decomposition theorem and its fixed-base
ergodic-generator corollary. In Section~4, we prove
Theorem~\ref{thm:treeable-rich}, constructing a strongly increasing
exhaustion in the non-hyperfinite case. Finally, Section~5 proves
Theorem~\ref{thm:rich_and_non_rich}: the Bernoulli action of
$\Z\times\F_r$ provides the non-richly ergodic example, while the product
of the dyadic odometer with the Bernoulli $\F_r$-action is shown to be
richly ergodic by expressing its orbit relation as an increasing union
of ergodic treeable subrelations.

\subsection*{Acknowledgements}

This paper is based on part of the author's Ph.D. thesis \cite{Dai}, written under the supervision of Clinton Conley. The exposition and proofs have been revised to give a self-contained presentation of the results on rich ergodicity. The author
is grateful to Clinton Conley for guidance, encouragement, and
many helpful discussions.

\section{Preliminaries}

\subsection{The space of subequivalence relations}

For a p.m.p. countable Borel equivalence relation (CBER) $E$ on $(X,\mu)$,
its full group is
\[
 [E]=\{T\in\operatorname{Aut}(X,\mu):\operatorname{graph}(T)\subseteq E\},
\]
where $\operatorname{Aut}(X,\mu)$ is the group of measure-preserving
automorphisms of $(X,\mu)$. The full group carries the uniform metric
\[
 d_u(S,T)=\mu(\{x:S(x)\neq T(x)\}).
\]
Let
\[
 \cS(E)=\{F:F\text{ is a p.m.p. CBER and }F\subseteq E\}.
\]
The map $F\mapsto[F]$ identifies $\cS(E)$ with a closed subspace of the
Wijsman space of closed subgroups of $([E],d_u)$
\cite[Proposition~4.1]{Kechris}. The resulting topology is called the
weak topology. For $T\in[E]$ and $F\in\cS(E)$, put
\[
 A_{T,F}=\{x:(x,T(x))\in F\}.
\]
Kechris's distance formula \cite[Proposition~4.2]{Kechris} is
\begin{equation}\label{eq:distance}
 d_u(T,[F])=1-\mu(A_{T,F}).
\end{equation}
Thus $F_n\to F$ weakly exactly when the quantities in
\eqref{eq:distance} converge for every $T\in[E]$.

There is also a strong topology defined using a uniquely generating
sequence for $E$. The weak and strong topologies agree
\cite[Theorem~4.15]{Kechris}; we call their common topology the
\emph{canonical Polish topology} on $\cS(E)$. We will use the monotone
convergence property
\[
 F_n\uparrow F\quad\text{or}\quad F_n\downarrow F
 \quad\Longrightarrow\quad F_n\longrightarrow F
\]
from \cite[Proposition~4.6]{Kechris}.

\subsection{Rich ergodicity, non-approximability, and perpendicularity}

\begin{definition}
For an ergodic p.m.p. CBER $E$, set
\[
 \operatorname{ERG}_E=\{F\in\cS(E):F\text{ is ergodic}\}.
\]
The relation $E$ is \emph{richly ergodic} if $\operatorname{ERG}_E$ is
dense in $\cS(E)$.
\end{definition}

The set $\operatorname{ERG}_E$ is $G_\delta$
\cite[Theorem~10.1]{Kechris}, so density is equivalent to comeagerness.

\begin{definition}
An ergodic p.m.p. CBER $E$ is \emph{non-approximable} if every increasing
exhaustion $F_0\subseteq F_1\subseteq\cdots$ with $E=\bigcup_nF_n$ is
trivial in the following sense: for some $n$ and some positive-measure
Borel set $B$,
\[
 F_n\mathbin{\upharpoonright}B=E\mathbin{\upharpoonright}B.
\]
\end{definition}

\begin{proposition}[Kechris {\cite[Proposition~11.3]{Kechris}}]
\label{prop:nonapprox-not-rich}
If $E$ is ergodic and non-approximable, then $E$ is not richly ergodic.
\end{proposition}

For subequivalence relations $R,S$ on $X$, write $R\vee S$ for the
smallest equivalence relation containing both $R$ and $S$.

\begin{definition}
For p.m.p. CBERs $R,S$ on $X$, write $R\perp S$ if there is no reduced
alternating cycle
\[
 x_0\,R\,x_1\,S\,x_2\,R\,\cdots\,S\,x_{2m}=x_0,
 \qquad m\geq1,
\]
with consecutive vertices distinct. Equivalently, $R\vee S$ is the
free product $R*S$.
\end{definition}

Perpendicularity is stronger than transversality: $R\perp S$ implies
$R\cap S=\operatorname{id}_X$, but the converse need not hold.

\begin{definition}
Let $E_0\subseteq E_1\subseteq\cdots$ and $E=\bigcup_nE_n$. The sequence
$(E_n)$ is \emph{strongly increasing} if, for every $n$, there is an
ergodic $T\in[E]$ such that
\[
 E_n\perp E_T.
\]
Here $E_T$ is the orbit equivalence relation generated by $T$.
\end{definition}

\begin{proposition}[Kechris {\cite[Proposition~11.4]{Kechris}}]
\label{prop:strongly-increasing}
If $(E_n)$ is strongly increasing and $E=\bigcup_nE_n$, then $E$ is
richly ergodic.
\end{proposition}

We also use the following two facts.

\begin{proposition}[Kechris {\cite[Proposition~11.6]{Kechris}}]
\label{prop:rich-union}
If $E_0\subseteq E_1\subseteq\cdots$ are richly ergodic, then
$E=\bigcup_nE_n$ is richly ergodic.
\end{proposition}

\begin{proposition}[Kechris {\cite[Proposition~11.7]{Kechris}}]
\label{prop:hyperfinite-rich}
Every ergodic hyperfinite p.m.p. CBER is richly ergodic.
\end{proposition}

\section{Cost and treeable CBERs}

All graphs below are symmetric, irreflexive, locally countable Borel
graphs; any loops arising in edge-sliding constructions are discarded.
A graph $G$ on $X$ is \emph{p.m.p.} if its connectedness relation
$E_G$ preserves $\mu$. It is a \emph{graphing} of $E$ if $E_G=E$.
An acyclic graphing is a \emph{treeing}, and a relation admitting a
treeing is \emph{treeable}. A relation is \emph{hyperfinite} if it is
an increasing union of finite Borel equivalence relations; equivalently,
it is generated by a Borel automorphism \cite{KM}. A graph is called
ergodic or hyperfinite when its connectedness relation has that property.

The \emph{cost} of a p.m.p. graph $G$ is
\[
 C_\mu(G)=\frac12\int_X\deg_G(x)\,d\mu(x),
\]
and the cost of a p.m.p. CBER $E$ is
\[
 C_\mu(E)=\inf\{C_\mu(G):G\text{ is a Borel graphing of }E\}.
\]
We use the same definition for the restriction of $\mu$ to a Borel set,
without normalizing. A fundamental theorem of Gaboriau states that
\emph{treeings attain cost}: if $G$ is a treeing of $E$, then
$C_\mu(G)=C_\mu(E)$ \cite[Theorem~IV.1]{GaboriauCost}.

We next recall the compression formula, which increases the
normalized cost upon restriction to a small positive-measure
set when $C_\mu(E)>1$.

\begin{theorem}[Gaboriau {\cite[Proposition~II.6]{GaboriauCost}}]
\label{thm:cost-compression}
Let $E$ be a p.m.p. CBER on $(X,\mu)$, and let $S\subseteq X$ be a
Borel complete section for $E$. Then
\[
 C_\mu(E)=C_{\mu\upharpoonright S}(E\upharpoonright S)+\mu(X\setminus S),
\]
where $\mu\upharpoonright S$ denotes the unnormalized restriction.
\end{theorem}

For a positive-measure Borel set $S$, write
$\mu_S(B)=\mu(B)/\mu(S)$ for Borel $B\subseteq S$.

\begin{corollary}\label{cor:zoom_in_cost}
If $E$ is ergodic, then for every positive-measure Borel set $S\subseteq X$,
\[
 C_{\mu_S}(E\upharpoonright S)
 =1+\frac{C_\mu(E)-1}{\mu(S)}.
\]
In particular, if $C_\mu(E)>1$, these normalized costs become arbitrarily
large as $\mu(S)$ tends to zero.
\end{corollary}

\begin{proof}
Ergodicity implies that the $E$-saturation of $S$ is conull. Apply
Theorem~\ref{thm:cost-compression} on this saturation and use
$C_{\mu\upharpoonright S}(E\upharpoonright S)
=\mu(S)C_{\mu_S}(E\upharpoonright S)$.
\end{proof}

We also recall the restriction property of treeability.

\begin{theorem}[Jackson--Kechris--Louveau
{\cite[Proposition~3.3(i)]{JKL}}]\label{thm:tree_restriction}
If $E$ is treeable and $A\subseteq X$ is Borel, then
$E\upharpoonright A$ is treeable.
\end{theorem}

We use \emph{well-iterated edge slide} in the sense of
\cite[Section~3.C]{MT}. Such slides preserve connectedness and preserve
acyclicity when the original graph is acyclic
\cite[Propositions~3.13(b) and~3.20(b)]{MT}.

The following theorem provides an ergodic hyperfinite subgraph after
edge sliding.

\begin{theorem}[Miller--Tserunyan
{\cite[Theorem~11.1]{MT}}]\label{thm:edge_slide_hyperfinite}
Let $G$ be a locally countable ergodic p.m.p. Borel graph on $(X,\mu)$.
There is a Borel well-iterated edge slide $\widetilde G$ of $G$
containing an ergodic hyperfinite Borel subgraph.
\end{theorem}

We next recall the ergodic hyperfinite decomposition theorem and its
relative form. Following Miller--Tserunyan, a graph $G$
\emph{supergraphs} a relation $E$ if $G\cap E$ graphs $E$. A sequence
$(E_j)$ is \emph{increasingly transverse} if
\[
 E_j\cap\bigvee_{i<j}E_i=\operatorname{id}_X
 \qquad (j\geq1).
\]
A slide based on a relation
$E_0$ fixes the edges in $G\cap E_0$.

\begin{theorem}[Miller--Tserunyan
{\cite[Theorem~6.14]{MT}}]\label{thm:edge_slide_acyclic_graph}
Let $G$ be a graphing of an ergodic p.m.p. CBER $E$. There are a Borel
well-iterated edge slide $\widetilde G$ of $G$ and an increasingly
transverse decomposition
\[
 E=\bigvee_{j=0}^{N}E_j,
 \qquad N\in\N\cup\{\infty\},\quad N\leq C_\mu(G),
\]
such that $\widetilde G$ supergraphs every $E_j$, each $E_j$ with
$0\leq j<N$ is ergodic hyperfinite, and $E_N$ is finite if $N<\infty$.
When $N=\infty$, the join is over all $j\in\N$ and there is no final
finite factor.
\end{theorem}

For the proof of our main theorem, the essential point is that an
already chosen ergodic base can remain fixed.

\begin{theorem}[Miller--Tserunyan
{\cite[Theorem~6.13]{MT}}]\label{thm:relative-ergodic-decomposition}
Let $E_0\subseteq E$ be ergodic p.m.p. CBERs, and let $G$ be a graphing
of $E$ that supergraphs $E_0$. Then $G$ admits a Borel $E_0$-based
well-iterated edge slide $\widetilde G$ and an increasingly transverse
decomposition
\[
 E=\bigvee_{j=0}^{N}E_j,
 \qquad N\in\N\cup\{\infty\},\quad N\leq C_\mu(G),
\]
with the prescribed first factor $E_0$, where $\widetilde G$ supergraphs
every $E_j$, the factors $E_j$ for $1\leq j<N$ are ergodic hyperfinite,
and $E_N$ is finite if $N<\infty$.
\end{theorem}

If $G$ is a treeing in either theorem, then $\widetilde G$ is a treeing
and the decomposition is free. Indeed, the graphs
$\widetilde G\cap E_j$ are pairwise disjoint subgraphs of a common
treeing, and each graphs its corresponding factor. Thus no reduced
cycle can alternate between the factors.

For an aperiodic automorphism $T$, we write
\[
 \mathcal T
 :=\{(x,T(x)):x\in X\}\cup\{(T(x),x):x\in X\}
\]
for its undirected graph, so that $E_{\mathcal T}=E_T$. The same
convention applies to first-return transformations.

\begin{corollary}\label{cor:s3-relative-generator}
Let $E$ be a p.m.p. CBER on $(X,\mu)$, and let $T\in[E]$ be ergodic.
If $E$ admits a treeing containing $\mathcal T$ and $C_\mu(E)>2$, then
there is an ergodic $S\in[E]$ such that $E_T\perp E_S$.
\end{corollary}

\begin{proof}
The relation $E$ is ergodic because it contains $E_T$. Apply
Theorem~\ref{thm:relative-ergodic-decomposition} with prescribed base
$E_0=E_T$. The decomposition is free and preserves this base. There
must be an additional ergodic hyperfinite factor: otherwise
$E=E_T*F$ for a finite relation $F$, and
\[
 C_\mu(E)\leq 1+C_\mu(F)
 =2-\int_X\frac{1}{|[x]_F|}\,d\mu(x)<2,
\]
contrary to the hypothesis. Here we use the standard cost formula for
finite relations \cite[Proposition~III.3]{GaboriauCost}. An automorphism
$S$ generating an additional ergodic hyperfinite factor is ergodic and
satisfies $E_T\perp E_S$.
\end{proof}

\section{Treeable ergodic equivalence relations are richly ergodic}
\label{sec:treeable-richly-ergodic}

After treating the hyperfinite case, we prove
Theorem~\ref{thm:treeable-rich} by constructing a strongly increasing
exhaustion.

\begin{theorem}
\label{thm:s4-treeable-richly-ergodic}
Let $E$ be an ergodic treeable p.m.p. CBER on $(X,\mu)$. Then $E$
is richly ergodic.
\end{theorem}

\begin{proof}
Let $G$ be a Borel treeing of $E$. By
Theorem~\ref{thm:edge_slide_hyperfinite}, after replacing $G$
by a well-iterated edge slide, we may assume that $G$ contains
an ergodic hyperfinite Borel subgraph. Since these slides
preserve connectedness and acyclicity, $G$ remains a treeing
of $E$. Replacing this subgraph by the undirected graph of an
ergodic automorphism generating the same relation also
preserves the treeing, since its free-product position
relative to the remaining edges is unchanged. Thus, keeping
the notation $G$ for the resulting treeing, we may assume
without loss of generality that
\[
G=\mathcal T\sqcup H,
\]
where $T\in[E]$ is ergodic. Since treeings attain cost,
\[
C_\mu(E)=C_\mu(G)=1+C_\mu(H).
\]
If $C_\mu(H)=0$, then $E=E_T$ modulo null sets. In this case $E$ is
hyperfinite and is richly ergodic by Proposition~\ref{prop:hyperfinite-rich}. Thus we may assume
that $C_\mu(H)>0$.

By Proposition~\ref{prop:strongly-increasing}, it suffices to construct a strongly increasing
exhaustion of $E$. Choose a Borel partition
\[
X=\bigsqcup_{n\in\N}A_n
\]
such that, for every $n\in\N$,
\[
\mu(A_n)>0,
\qquad C_\mu\bigl(H\cap(A_n\times A_n)\bigr)>0.
\]
To see that such a partition exists, write $H$ as a countable union of
graphs of partial Borel isomorphisms. Since $C_\mu(H)>0$, one of these
isomorphisms, say $\varphi:D\to\varphi(D)$, has a domain of positive
measure. To make its domain and range disjoint, choose a member $B$ of a
countable Borel family, closed under complements and separating points,
such that
$D\cap B\cap\varphi^{-1}(X\setminus B)$ has positive measure, and
restrict $\varphi$ to this set. Thus we may assume
$D\cap\varphi(D)=\varnothing$.
Partition $D$ into Borel sets $D_n$, $n\in\N$, of positive
measure. The sets $D_n\cup\varphi(D_n)$ are pairwise disjoint; extend
them to a partition $(A_n)_{n\in\N}$ of $X$ by adding the
remaining points to $A_0$. The edges of $\varphi\upharpoonright D_n$
then witness the required positive cost inside $A_n$.

For $n\in\N$, put
\[
\begin{aligned}
X_n&:=\bigsqcup_{m<n}A_m,
&\qquad H_n&:=H\cap(X_n\times X_n),\\
G_n&:=\mathcal{T}\sqcup H_n,
&\qquad E_n&:=E_{G_n}.
\end{aligned}
\]
Then $G_n\uparrow G$. Every finite $G$-path lies in some $G_n$, so
\[
E_0\subseteq E_1\subseteq\cdots,
\qquad \bigcup_{n\in\N}E_n=E.
\]
Fix $n$. We will find an ergodic $S\in[E]$ such that $E_n\perp E_S$.
Let
\[
R_n:=H\cap(A_n\times A_n),
\qquad G':=\mathcal{T}\sqcup R_n.
\]
Then $R_n\cap H_n=\varnothing$ and $C_\mu(R_n)>0$. Moreover, $G'$
is a treeing of an ergodic subequivalence relation of $E$, and
\[
C_\mu(E_{G'})=C_\mu(G')=1+C_\mu(R_n)>1.
\]
By Corollary~\ref{cor:zoom_in_cost}, we may choose a Borel set $Y\subseteq A_n$ of sufficiently small positive
measure so that
\[
C_{\mu_Y}(E_{G'}\upharpoonright Y)
=1+\frac{C_\mu(E_{G'})-1}{\mu(Y)}>2,
\]
where $\mu_Y(B)=\mu(B)/\mu(Y)$ for Borel $B\subseteq Y$. Here $Y$
is a complete section modulo null sets, since $T$ is ergodic.

We work on $(Y,\mu_Y)$ and consider the first-return transformation
\[
r_Y(y):=\min\{k\geq 1:T^k(y)\in Y\},
\qquad T_Y(y):=T^{r_Y(y)}(y).
\]
The transformation $T_Y$ is p.m.p. and ergodic. To apply
Corollary~\ref{cor:s3-relative-generator}, we check that
$E_{G'}\upharpoonright Y$ admits a treeing containing $\mathcal{T}_Y$.
After discarding an $E$-invariant null set, every $T$-orbit visits
$Y$ infinitely often in both directions, so the intervals
\[
I_y:=\{T^j(y):0\leq j<r_Y(y)\},\qquad y\in Y,
\]
partition $X$. Let $p:X\to Y$ send each $I_y$ to $y$. Contract these
intervals in $G'$ and discard loops; explicitly, the resulting graph is
\[
G'_Y:=\bigl\{(p(x),p(x')):(x,x')\in G',\ p(x)\neq p(x')\bigr\}.
\]
The map $p$ is Borel and finite-to-one, so $G'_Y$ is a locally countable
Borel graph. Contracting disjoint connected subtrees cannot create a
cycle. Also, $p(x)\mathrel{E_{G'}}x$, and paths project to walks after
consecutive repetitions are removed. Thus $G'_Y$ is a treeing of
$E_{G'}\upharpoonright Y$.
The edge from $T^{r_Y(y)-1}(y)$ to $T_Y(y)$ becomes the edge from $y$
to $T_Y(y)$, so $\mathcal{T}_Y\subseteq G'_Y$.
Corollary~\ref{cor:s3-relative-generator} therefore gives an ergodic
$S_Y\in[E_{G'}\upharpoonright Y]$ such that
\[
E_{T_Y}\perp E_{S_Y}.
\]

For $k\geq 1$, put
\[
U_k:=S_Y^kT_YS_Y^{-k},
\]
and define the Borel graph
\[
H_Y:=\bigcup_{k\geq 1}
\left(\{(y,U_k(y)):y\in Y\}
\cup\{(U_k(y),y):y\in Y\}\right).
\]
Returning to $X$, put
\[
\begin{split}
F_n&:=E_{\mathcal{T}\sqcup H_Y},\\
G'_n&:=\mathcal{T}\sqcup H_n\sqcup H_Y=G_n\sqcup H_Y,
\qquad E'_n:=E_{G'_n}.
\end{split}
\]

\begin{claim}
\label{clm:s4-conjugate-amplification}
The graph $\mathcal{T}\sqcup H_Y$ is acyclic,
$C_\mu(F_n)=\infty$, and $E_{H_n}\perp F_n$.
Consequently, $G'_n$ is a treeing of $E'_n$ and
$C_\mu(E'_n)=\infty$.
\end{claim}

\begin{proof}[Proof of the claim]
Both $T_Y$ and $S_Y$ are aperiodic. Together with
$E_{T_Y}\perp E_{S_Y}$, this implies that they generate a free
$\F_2$-action on a conull set. Set $U_0:=T_Y$ for this paragraph.
By a standard fact in group theory, the family $(U_k)_{k\geq 0}$ freely generates a subgroup isomorphic to $\F_\infty$. Consequently, $\mathcal{T}_Y\sqcup H_Y$ is acyclic, and the
undirected graphs of the $U_k$, $k\geq 0$, are pairwise edge-disjoint.

Replace every $\mathcal{T}_Y$-edge joining $y$ to $T_Y(y)$ by the
corresponding $\mathcal{T}$-path
\[
y,\ T(y),\ \ldots,\ T_Y(y).
\]
These paths have pairwise disjoint interiors in $X\setminus Y$ and
cover the $T$-orbits modulo null sets. Thus this replacement only
subdivides edges of $\mathcal{T}_Y\sqcup H_Y$, and the resulting graph
is $\mathcal{T}\sqcup H_Y$. In particular, it is acyclic.

Each $U_k$ is aperiodic on $Y$, so its undirected graph has cost
$\mu(Y)$ with respect to the unnormalized measure $\mu\upharpoonright Y$.
Since these graphs are edge-disjoint and treeings attain cost, we obtain
\[
C_\mu\bigl(E_{\mathcal{T}\sqcup H_Y}\bigr)
=C_\mu(\mathcal{T}\sqcup H_Y)
=1+\sum_{k\geq 1}\mu(Y)
=\infty.
\]

Every $H_Y$-edge belongs to $E_{G'}$, hence $F_n\subseteq E_{G'}$.
The graphs $H_n$ and $G'$ are disjoint subgraphs of the treeing $G$,
so their generated relations are perpendicular. It follows that
\[
E_{H_n}\perp E_{G'}
\qquad\text{and consequently}\qquad
E_{H_n}\perp F_n.
\]

The graph $H_n$ is acyclic, so its perpendicularity to $F_n$ implies
that $G'_n$ is a treeing of $E'_n$. Since $G'_n$ contains
$\mathcal{T}\sqcup H_Y$ and treeings attain cost,
\[
C_\mu(E'_n)=C_\mu(G'_n)=\infty.
\]
\end{proof}

In particular,
\[
E_n\subseteq E'_n\subseteq E.
\]

Finally, apply Corollary~\ref{cor:s3-relative-generator} to $F_n$,
whose treeing $\mathcal{T}\sqcup H_Y$ contains $\mathcal{T}$ and
has infinite cost. This gives an ergodic $S\in[F_n]$ satisfying
$E_T\perp E_S$. Since $E_{H_n}\perp F_n$ and
$E_T\vee E_S\subseteq F_n$, associativity of free products gives
\[
\begin{split}
E_n\vee E_S
&=E_{H_n}*(E_T*E_S)\\
&=(E_{H_n}*E_T)*E_S
=E_n*E_S.
\end{split}
\]
Thus $E_n\perp E_S$, and $S\in[E]$ because $F_n\subseteq E$.
As $n$ was arbitrary, $(E_n)_{n\in\N}$ is strongly increasing.
Proposition~\ref{prop:strongly-increasing} now implies that $E$ is richly ergodic.
\end{proof}

\section{Groups with richly ergodic and non-richly ergodic actions}

Throughout this section, fix $2\leq r<\infty$ and write
$\F_r=\langle s_1,\ldots,s_r\rangle$. We prove
Theorem~\ref{thm:rich_and_non_rich}.

\subsection{A non-rich Bernoulli action}
The following example follows immediately from the
non-approximability theorem of Gaboriau--Tucker-Drob
\cite[Theorem~1.1]{GTD}.
\begin{proposition}\label{prop:non-rich-action} 
The Bernoulli action of $\Z\times\F_r$ with a nontrivial standard
probability base is free and ergodic, and its orbit relation is not
richly ergodic.
\end{proposition}

\begin{proof}
The action is free and ergodic. Its $\F_r$-restriction is a Bernoulli
action of a nonamenable group, hence strongly ergodic, while its
$\Z$-restriction is a Bernoulli transformation, hence ergodic.
The two subgroups are infinite, finitely generated, and commute, so
\cite[Theorem~1.1]{GTD} implies that the orbit relation is
non-approximable. By Proposition~\ref{prop:nonapprox-not-rich}, it is
not richly ergodic.
\end{proof}

\subsection{A richly ergodic equivalence relation}
In this subsection, we construct a free ergodic p.m.p. action of $\Z\times\F_r$
whose orbit relation is an increasing union of ergodic
treeable subrelations, and hence is richly ergodic.

Write $\Z_2=\Z/2\Z$. We begin with two basic actions.
\begin{enumerate}
\item \textbf{The odometer action.}
Let $X=\{0,1\}^{\N}$ with the coin-flip product probability measure
$\mu_1$, and let $T:X\to X$ be the dyadic odometer, given by addition
of $1$ in binary with carry starting at coordinate $0$.
\item \textbf{The Bernoulli shift.}
Let $Y\subseteq2^{\F_r}$ be the free part of the Bernoulli shift of
$\F_r$, equipped with the coin-flip product probability measure $\mu_2$.
\end{enumerate}
Define the free ergodic p.m.p. product action
\[
 \Z\times\F_r\mathrel{\overset{\alpha}{\curvearrowright}}
 (X\times Y,\mu_1\otimes\mu_2)
\]
by
\[
 \alpha(m,\gamma)(x,y)=(T^m x,\gamma\cdot y),
 \qquad m\in\Z,\quad\gamma\in\F_r.
\]
For an action $\rho$, write $E_\rho$ for its orbit relation. We will
express $E_\alpha$ as an increasing union of ergodic treeable relations
$E_{\beta_n}$, where $\beta_n$ is an action of $(\Z_2)^n\times\F_r$.

Choose a countable Borel partition
\[
 Y=\bigsqcup_{n\in\N}Y_n,
 \qquad \mu_2(Y_n)>0\quad(n\in\N),
\]
and define $C:Y\to\N$ by $C(y)=n$ exactly when $y\in Y_n$.
For each $n\in\N$, let $\sigma_n:X\to X$ flip the first $n$ coordinates:
\[
 (\sigma_n x)_i=
 \begin{cases}
  1-x_i,&i<n,\\
  x_i,&i\geq n.
 \end{cases}
\]
In particular, $\sigma_0=\operatorname{id}_X$.
Define an action $\beta_0$ of $\F_r$ by
\begin{align*}
 T_{s_1}^{\beta_0}(x,y)&=(\sigma_{C(y)}x,s_1\cdot y),\\
 T_{s_j}^{\beta_0}(x,y)&=(x,s_j\cdot y),\qquad 2\leq j\leq r.
\end{align*}
These transformations define a free p.m.p. action, since they preserve
$\mu_1\otimes\mu_2$ and project to the free action on $Y$.

\begin{lemma}\label{lem:ergodicity}
The action $\beta_0$ is ergodic.
\end{lemma}

\begin{proof}
Let $D\subseteq X\times Y$ be invariant. Since $s_2$ acts ergodically
on $Y$ and fixes $X$, there is a Borel set $A\subseteq X$ such that
$D=A\times Y$ modulo a null set. Invariance under $T_{s_1}^{\beta_0}$
gives
\[
 0=(\mu_1\otimes\mu_2)
   \bigl((T_{s_1}^{\beta_0})^{-1}(D)\triangle D\bigr)
  =\sum_{n\in\N}
   \mu_1\bigl(\sigma_n^{-1}(A)\triangle A\bigr)\mu_2(Y_n).
\]
Thus $A$ is invariant under every $\sigma_n$. Since
$\sigma_{i+1}\sigma_i$ flips coordinate $i$, the tail zero--one law
gives $\mu_1(A)\in\{0,1\}$. Hence $D$ is null or conull.
\end{proof}

For each $n\in\N$, let $(\Z_2)^n$ act by independent flips of the
first $n$ coordinates of $X$, trivially on $Y$. These flips commute
with $\beta_0$, so they define a free p.m.p. action
\[
 (\Z_2)^n\times\F_r\mathrel{\overset{\beta_n}{\curvearrowright}}
 (X\times Y,\mu_1\otimes\mu_2).
\]
The actions are compatible as $n$ increases, and each $\beta_n$ is
ergodic because $E_{\beta_0}\subseteq E_{\beta_n}$.

\begin{lemma}\label{lem:exhaustion}
The orbit relations satisfy
\[
 E_{\beta_0}\subseteq E_{\beta_1}\subseteq\cdots,
 \qquad \bigcup_{n\in\N}E_{\beta_n}=E_\alpha.
\]
\end{lemma}

\begin{proof}
The odometer orbit relation is eventual equality modulo null sets.
Since every generator of $\beta_n$ changes only finitely many coordinates
of $x$ and moves $y$ within its $\F_r$-orbit, we have
$E_{\beta_n}\subseteq E_\alpha$.
Conversely, suppose $(x,y)\mathrel{E_\alpha}(x',y')$ and choose
$\gamma\in\F_r$ with $\gamma\cdot y=y'$. Write
$\beta_0(\gamma)(x,y)=(\widetilde{x},y')$.
The points $\widetilde{x}$ and $x'$ differ in only finitely many
coordinates, so for some sufficiently large $N$, a flip in $(\Z_2)^N$
carries $\widetilde{x}$ to $x'$. Hence
$(x,y)\mathrel{E_{\beta_N}}(x',y')$, proving the reverse inclusion.
\end{proof}

\begin{proposition}\label{prop:rich-product-action}
The product action $\alpha$ of $\Z\times\F_r$ is free and ergodic,
and its orbit relation is richly ergodic.
\end{proposition}

\begin{proof}
Each $\beta_n$ is free and ergodic. The group $(\Z_2)^n\times\F_r$
is strongly treeable: strong treeability is preserved by extensions
with finite kernel, and $\F_r$ is strongly treeable
\cite[Theorem~B.3(4)]{CGMTD}. Thus every $E_{\beta_n}$ is treeable
and, by Theorem~\ref{thm:treeable-rich}, richly ergodic.
Lemma~\ref{lem:exhaustion} and Proposition~\ref{prop:rich-union}
now imply that $E_\alpha$ is richly ergodic.
\end{proof}

Together with Proposition~\ref{prop:non-rich-action}, this proves
Theorem~\ref{thm:rich_and_non_rich}.

The existence of a strongly increasing exhaustion is sufficient for
rich ergodicity by Proposition~\ref{prop:strongly-increasing}, whereas
non-approximability obstructs rich ergodicity by
Proposition~\ref{prop:nonapprox-not-rich}. We conclude with two questions
about the converses in the non-hyperfinite case.

\begin{question}\label{ques:rich-ergodicity-converses}
Does there exist a non-hyperfinite ergodic p.m.p. countable Borel
equivalence relation that is richly ergodic but does not admit a
strongly increasing exhaustion?
\end{question}

\begin{question}\label{ques:non-rich-approximable}
Does there exist a non-hyperfinite ergodic p.m.p. countable Borel
equivalence relation that is not richly ergodic but admits a
nontrivial approximation?
\end{question}

\end{document}